\documentclass[11pt,a4paper]{article}
\pdfoutput=1 
\usepackage{amsmath,amssymb,bigdelim,mathrsfs,multicol,multirow}
\usepackage[inline]{enumitem}
\usepackage[pdftex]{graphicx}
\usepackage[top=30truemm,bottom=30truemm,left=25truemm,right=25truemm]{geometry}
 
\usepackage{mathtools}

\numberwithin{equation}{section}

\usepackage[amsmath,hyperref,thref,thmmarks]{ntheorem}
\theoremstyle{plain}
\theoremseparator{.}
\theorembodyfont{\normalfont}

\newtheorem{proposition}{Proposition}[section]
\newtheorem{lemma}{Lemma}[section]
\newtheorem{theorem}{Theorem}[section]
\newtheorem{corollary}{Corollary}[section]

\newtheorem{remark}{Remark}[section]

\theoremheaderfont{\normalfont}
\theoremsymbol{\ensuremath{\blacksquare}}
\newtheorem*{proof}{\textit{Proof}}

\newcommand{\defeq}{\coloneqq}
\newcommand{\eqdef}{\eqqcolon}

\newcommand{\isom}{\cong}

\newcommand{\tensor}{\otimes}

\newcommand{\Kernel}{\, \mathrm{Ker} \,}
\newcommand{\Image}{\, \mathrm{Im} \,}

\newcommand{\QRational}{\mathbb{Q}}
\newcommand{\Real}{\mathbb{R}}

\newcommand{\ZInteger}{\mathbb{Z}}

\newcommand{\setmid}{ \ \middle| \ }
\newcommand{\restr}[3][]{{\left. {#2} \right|_{#3}}}

\newcommand{\id}{\mathrm{id}}

\newcommand{\setin}[3][]{{\left\{ {#2} \setmid {#3} \right\}}}

\newcommand{\Symp}{\mathrm{Sp}}
\newcommand{\symp}{\mathfrak{sp}}
\newcommand{\SymplTwogQ}{\Symp(2g; \QRational)}

\newcommand{\symplTwogQ}{\symp(2g; \QRational)}

\newcommand{\cgplus}{\mathfrak{c}_{g}^{+}}
\newcommand{\cg}{\mathfrak{c}_{g}}
\newcommand{\cinf}{\mathfrak{c}_{\infty}}

\newcommand{\symmprod}{\odot}

\newcommand{\transpose}[1]{ {\, {}^t {#1}} }

\newcommand{\contr}[2]{\mu_{{#1}{#2}}}
\newcommand{\contrend}{\contr{\mathrm{end}}{}}
\newcommand{\alter}[2]{\Lambda_{{#1}{#2}}}
\newcommand{\alterend}{\alter{\mathrm{end}}{}}

\title{The second homology group of the commutative case of Kontsevich's symplectic derivation Lie algebra}
\author{Shuichi Harako}
\date{}
\begin{document}
\maketitle
\abstract{
    The symplectic derivation Lie algebras defined by Kontsevich 
    are related to various geometric objects
    including moduli spaces of graphs and of Riemann surfaces, 
    graph homologies, Hamiltonian vector fields, etc. 
    Each of them and its Chevalley-Eilenberg chain complex have 
    a \(\ZInteger_{\geq 0}\)-grading called weight. 
    We consider one of them \(\cg\), called the ``commutative case'', 
    and its positive weight part \(\cgplus \subset \cg\). 
    The symplectic invariant homology of \(\cgplus\) 
    is closely related to the commutative graph homology, 
    hence there are some computational results 
    from the viewpoint of graph homology theory. 
    However, the entire homology group \(H_\bullet (\cgplus)\) is not known well. 
    We determined \(H_2 (\cgplus)\) by using classical representation theory 
    of \(\SymplTwogQ\) and the decomposition by weight.  
}

\section{Introduction}
\par In \cite{Kontsevich1993,Kontsevich1994}, 
Kontsevich defined the three symplectic Lie algebras  
\(\mathfrak{l}_g, \mathfrak{a}_g, \mathfrak{c}_g\). 
They have deep relations to geometric objects below, 
therefore they are studied from various viewpoints. 
The Lie algebras \(\mathfrak{l}_g\) and \(\mathfrak{a}_g\) 
correspond to the moduli spaces of graphs and of Riemann surfaces respectively. 
The Lie algebra \(\cg\) is studied to describe 
invariants of 3-dimensional manifolds and 
some kinds of characteristic classes of foliations 
by an identification of the scalar extension of \(\cg\) 
with Hamiltonian vector fields satisfying certain conditions 
(e.g. \cite{Gelfand1992,Kotschick2009,Kontsevich1999}). 
\(\cg\) is also related to supergeometry and mathematical physics
through the map to the set of characteristic classes of \textit{Q-manifolds}, 
which is an important kind of supermanifolds \cite{Kontsevich1999,Qiu2010}. 
\par They are interpreted in terms of graph homology theory. 
We can take the direct limit of each Lie algebra 
e.g. \(\cinf = \lim_{g \to \infty} \cg\), 
and see its homology is endowed with the Hopf algebra structure. 
Kontsevich's theorem described the primitive part of the stable homology group 
of each Lie algebra as a certain kind of graph homology \cite[Theorem 1.1.]{Kontsevich1993}.
The homology of \(\cg\) is tied to 
the \textit{commutative graph homology} in the theorem. 
This version of graph homology is used in perturbative Chern-Simons theory 
and provides an extension of Vassiliev invariants \cite{BarNatan1995,KontsevichV}.  
\par The Lie algebra \(\cg\) has a \(\ZInteger_{\geq 0}\)-grading called \textit{weight}, 
and its positive weight part is denoted by \(\cgplus\).  
It is known that an argument using a spectral sequence shows that 
\begin{equation*}
    H_\bullet (\cg) \isom 
    H_\bullet (\symplTwogQ) \tensor H_\bullet (\cgplus)^\Symp
\end{equation*}
holds in the stable range.
Here \(H_\bullet (\cgplus)^\Symp\) is the \(\Symp\)-invariant part 
of \(H_\bullet (\cgplus)\). 
This isomorphism makes the computation of \(H_\bullet (\cg)\) relatively easier.
Moreover, it is one way 
to systematically construct cohomology classes of higher degree in \(\cg\) 
from ones of lower degree by taking duals and cup products. 
This method of taking positive weight part is applied to 
\(\mathfrak{l}_g\) and \(\mathfrak{a}_g\) in \cite{Morita2008}.
\par Kontsevich's theorem shows a relationship between 
\(H_\bullet (\cg)\) and the commutative graph homology. 
However, the problem of the computation of these groups still remains. 
There are some computational results from the viewpoint of graph (co)homology theory 
\cite{Garoufalidis1997,Conant2007,BarNatanDraft}.
Willwacher and \v{Z}ivkovi\'{c} gave the generating function of 
the Euler characteristic of the commutative graph homology 
and displayed it up to weight 60 in 
\cite[p.~575, Table~1]{Willwacher2015} 
as \(\chi_b^{\mathrm{odd}}\). 
The commutative graph homology itself is determined up to weight 12 
by Conant, Gerlits, and Vogtmann \cite{Conant2005}. 
\par 
The homology group \(H_\bullet (\cgplus)\) is a direct sum of 
the subspaces \(H_\bullet (\cgplus)_w\) generated by homogeneous elements of weight \(w\). 
In this paper, we prove the following. 
\begin{theorem}\label{main theorem}
    \(H_2 (\cgplus)_w = 0\) if \(g,w \geq 4\).
\end{theorem}
It is easy to see that \(H_1 (\cgplus) = S^3 \QRational^{2g}\) (see \thref{prop H1}), 
however, it has yet to be known about the higher degree of \(H_{\bullet} (\cgplus)\). 
\thref{main theorem} is proved in Section \ref{Proof of main theorem} 
by an application of classical representation theory and weight. 
The parts of weight 1--3 are also determined in terms of \(\Symp\)-modules, 
in conclusion, so is \(H_2 (\cgplus)\) as a corollary. 
\\
\par In Section \ref{The lie algebra cgplus}, 
we first recall the Lie algebra \(\cg\) and \(\cgplus\). 
\par In Section \ref{Representation theory of SymplTwogQ}-\ref{Detecting the highest weight vector}, 
we review the classical representation theory of \(\SymplTwogQ\). 
\par In Section \ref{Proof of main theorem}, 
we prove \thref{main theorem} by using the facts 
proved or introduced in Section \ref{Representation theory of SymplTwogQ}-\ref{Detecting the highest weight vector}. 
\par In Section \ref{Lower weight cases}, 
we give some corollaries of \thref{main theorem} 
in the case that weight is 1--3. 
Here we obtain the description of entire \(H_2 (\cgplus)\). 
\par We use the word ``\(\Symp\)'' as a shorthand for \(\SymplTwogQ\) 
when no confusion can arise. 
\subsubsection*{Acknowledgments}
The author is deeply grateful to his supervisor Takuya Sakasai for 
offering a lot of helpful advice. 
The author also wishes to express appreciation to his colleagues in his laboratory 
for constructive discussions and considerable encouragements. 
\section{The Lie algebra \(\cgplus\)} \label{The lie algebra cgplus}
We assume \(g \geq 4\).
Let \(H \defeq \QRational^{2g}\) be 
the fundamental representation of the symplectic group \(\SymplTwogQ\), and 
\(\mu \colon H \tensor H \to \QRational \) be
the symplectic form on \(H\).
We fix a symplectic basis \(a_1, \dots, a_g, b_1, \dots, b_g\) with respect to \(\mu\)
, i.e. for any \(1 \leq i, j \leq g\), \(\mu(a_i, a_j) = \mu(b_i, b_j) = 0\) and
\(\mu(a_i, b_j) = - \mu(b_j, a_i) = \delta_{ij}\) where \(\delta_{ij}\) is the Kronecker delta. 
\par For \(w \geq 0\), define 
\begin{equation}
\cg(w) \defeq S^{w + 2} H 
\qquad \text{and} \qquad 
\cgplus \defeq \bigoplus_{w \geq 1} \cg(w)
\subset \bigoplus_{w \geq 0} \cg(w) \eqdef \cg.
\end{equation}
An element \(\xi \in \cg(w) \subset \cg\) is said to be 
of \textit{weight} \(w\). 
\par We define a Lie bracket \(\left[, \right]\) on \(\cg\) by 
\begin{equation*}
\includegraphics{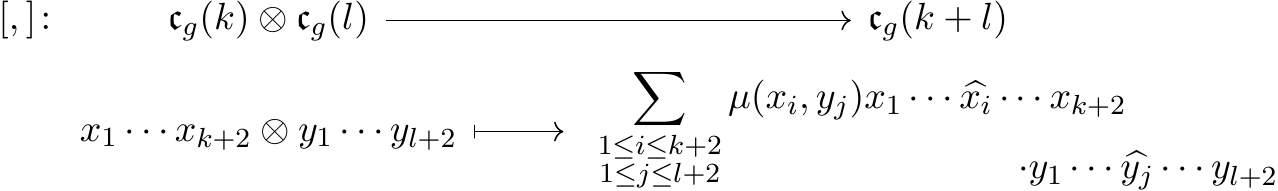}
\end{equation*}
for \(k, l \geq 0\) and monomials 
\(x_1 \cdots x_{k+2} \in \cg(k)\), \(y_1 \cdots y_{l+2} \in \cg(l)\). 
Here \(\widehat{\phantom{x_i}}\) means the absence of that component. 
This bracket is interpreted as the classical Poisson bracket on \(C^\infty (\Real^{2g})\) 
restricted to polynomial functions whose coefficients in \(\QRational\). 
Note that this bracket preserves weights on \(\cg\) and is \(\Symp\)-equivariant. 
Then the subspace \(\cgplus \subset \cg\) becomes a Lie subalgebra. 
\par The differential \(\partial_\bullet\) on the Chevalley-Eilenberg chain complex \(\wedge^\bullet \cg\) 
with respect to this Lie bracket is described as 
\begin{equation*}
    \partial_n (f_1 \wedge \cdots \wedge f_n)
    = \sum_{1 \leq i < j \leq n} (-1)^{i+j+1} \left[ f_i, f_j \right]
    \wedge f_1 \wedge \cdots \wedge \widehat{f_i} \wedge \cdots 
    \wedge \widehat{f_j} \wedge \cdots \wedge f_n
\end{equation*}
for any \(f_1 ,\dots, f_n \in \cg\). 
We also introduce weights on \(\wedge^\bullet \cg\) by declaring 
that \(f_1 \wedge \cdots \wedge f_n\) is of \textit{weight} \(k_1 + \cdots + k_n\) 
if \(f_1 \in \cg(k_1), \dots, f_n \in \cg(k_n)\). 
We denote by \(\left( \wedge^n \cg \right)_w\) 
the subspace of \(\wedge^n \cg\) generated by elements of weight \(w\). 
Similarly, we introduce a notation \(\left( \wedge^n \cgplus \right)_w \subset \wedge^n \cgplus\).  
\par Since the bracket map \(\left[, \right]\) is \(\Symp\)-equivariant, 
so is the differential \(\partial_\bullet\). 
The weight on \(\wedge^\bullet \cg\) defined above 
is also preserved by \(\partial_\bullet\). 
Therefore for \(n \geq 1\), 
we can decompose the chain space \(\wedge^n \cgplus\) by weight 
and moreover into \(\Symp\)-irreducible components 
because each \(\left( \wedge^n \cgplus \right)_w\) is 
a finite dimensional \(\Symp\)-module. 
\par Hereafter we concentrate on \(\cgplus\) rather than \(\cg\). 
\begin{remark}
We regard the vector space \(\wedge^\bullet \cgplus\) as 
the quotient space of \(\bigoplus_{n \geq 1} \left( \cgplus \right)^{\tensor n}\).
Let 
\(\pi \colon \bigoplus_{n \geq 1} \left( \cgplus \right)^{\tensor n} \twoheadrightarrow \wedge^\bullet \cgplus\) 
be the quotient map. 
We write 
\begin{equation*}
    \cg(k_1) \wedge \cdots \wedge \cg(k_n) \defeq 
    \pi (\cg(k_1) \tensor \cdots \tensor \cg(k_n))
\end{equation*}
for \(n \geq 2\).
We easily see that 
\(\cg(k) \wedge \cg(l) = \cg(l) \wedge \cg(k) \isom \cg(k) \tensor \cg(l)\) 
if \(k,l \geq 1\) and \(k \neq l\). 
This notation might be unusual, however we use it for simplicity. 
For example, 
\begin{equation}
    \begin{aligned}
        \cg(2) \wedge \cg(1) \wedge \cg(1) &{}= 
        \left( 
            \begin{array}{c}
                \text{The vector subspace of } \wedge^3 \cgplus \text{ generated by } \\ 
                \text{the elements of the form } f_1 \wedge f_2 \wedge f_3 \\
                \text{for } f_1 \in \cg(2) ,\ f_2, f_3 \in \cg(1)
            \end{array} 
        \right) \\
        &{}\isom \cg(2) \tensor \left( \wedge^2 \cg(1) \right). 
    \end{aligned}
\end{equation}
\end{remark}
\par We use the following natural way 
to regard \(\cg(k)\) and \(\cg(k) \wedge \cg(l)\) 
as \(\Symp\)-submodules of \(H^{\tensor n}\)s. 
Consider an injective map 
\begin{equation}
\begin{aligned}
    \includegraphics{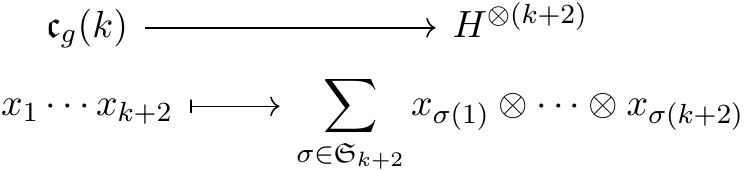}
\end{aligned}
\end{equation}
for \(k \geq 1\). 
We always see degree \(n\) polynomials of \(H\) 
as elements in \(H^{\tensor n}\) by this map. 
Then for \(k > l \geq 1\), we see 
\(\cg(k) \wedge \cg(l) \subset \wedge^2 \cgplus\) 
as a submodule of \(H^{\tensor (k+l+4)}\) by an injective map 
\begin{equation}
    \begin{aligned}
        \includegraphics{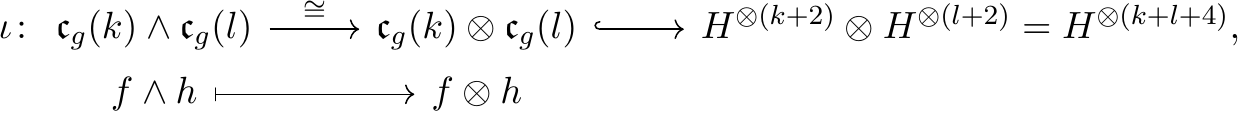}
    \end{aligned}
\end{equation}
and for \(k \geq 1\), 
\begin{equation}
    \begin{aligned}
        \includegraphics{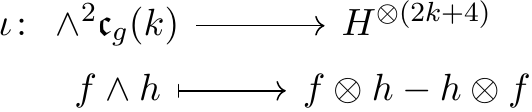}
    \end{aligned}
\end{equation}
using the same notation \(\iota\). 
For \(m,n \geq 1\), 
\(\alpha \in H^{\tensor m}\) and \(\beta \in H^{\tensor n}\), 
define 
\begin{equation}
    \alpha \wedge \beta \defeq 
        \alpha \tensor \beta - \beta \tensor \alpha
        \in H^{\tensor (m+n)}, 
    \qquad 
    \alpha \symmprod \beta \defeq 
        \alpha \tensor \beta + \beta \tensor \alpha
        \in H^{\tensor (m+n)}.
\end{equation}
\ \\
\par Let us refer to \(H_1 (\cgplus)\) before we discuss \(H_2 (\cgplus)\). 
\begin{proposition}\label{prop H1}
    \begin{itemize}
        \item[(1)] The map 
            \(\partial_2 = [,] \colon (\wedge^2 \cgplus)_w \to (\wedge^1 \cgplus)_w = \cg(w)\)
            is surjective if \(g, w\geq 2\). 
        \item[(2)] \(H_1 (\cgplus) = S^3 H\) if \(g \geq 2\). 
    \end{itemize}
\end{proposition}
\begin{proof}
    For (1), we see that this map is zero or surjective 
    because it is \(\Symp\)-equivariant 
    and \(\cg(w)=S^{w+2} H\) is \(\Symp\)-irreducible. 
    Now, we have 
    \(\partial_2 (a_1^w a_g \wedge a_1^2 b_g) = [a_1^w a_g, a_1^2 b_g] = a_1^{w+2}\). 
    This leads to surjectivity of this map. 
    \par (2) follows from (1) and \(\wedge^2 \cgplus = \bigoplus_{w \geq 2} (\wedge^2 \cgplus)_w\). 
\end{proof}
Contrary to this case, the chain space \((\wedge^2 \cgplus)_w\) 
is not \(\Symp\)-irreducible for general \(w\). 
In order to apply the similar method of confirming all the cycles are boundaries, 
we must know how the chain space \((\wedge^2 \cgplus)_w\) 
decomposes into \(\Symp\)-irreducible components. 
\section{Representation theory of \(\SymplTwogQ\)}
\label{Representation theory of SymplTwogQ}
The following is a classical theorem in representation theory 
(see e.g. \cite{Fulton1991}).
\begin{theorem}
    There is a bijection 
    \begin{equation*}
        \includegraphics{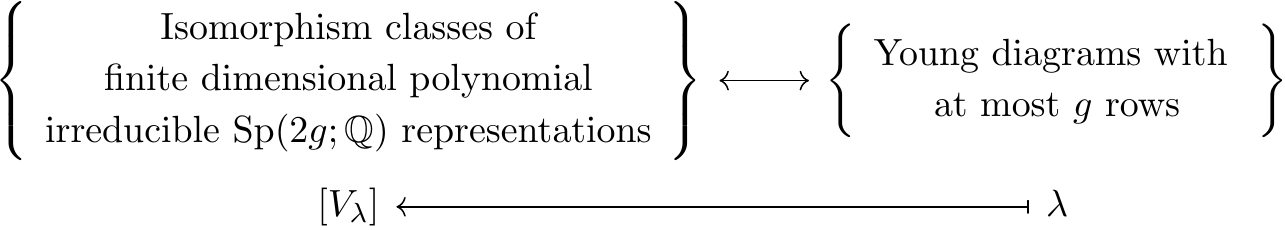}
    \end{equation*}
    where \(V_\lambda\) is a certain submodule of 
    \((\wedge^{\lambda^\prime_1} H ) \tensor \cdots \tensor (\wedge^{\lambda^\prime_d} H )\). 
    \(V_\lambda\) is generated by the element 
    \begin{equation*}
        a_\lambda \defeq (a_1 \wedge \cdots \wedge a_{\lambda^\prime_1}) \tensor \cdots 
        \tensor (a_1 \wedge \cdots \wedge a_{\lambda^\prime_d})
        \ \in (\wedge^{\lambda^\prime_1} H ) \tensor \cdots \tensor (\wedge^{\lambda^\prime_d} H )
    \end{equation*}
    as an \(\SymplTwogQ\)-module.
    Here 
    \(\transpose{\lambda} = [\lambda^\prime_1 \cdots \lambda^\prime_d] \quad
         (g \geq \lambda^\prime_1 \geq \cdots \geq \lambda^\prime_d \geq 1)\) 
    is the transpose of \(\lambda\), 
    and \(\left[ V_\lambda \right]\) is the isomorphism class of \(V_\lambda\). 
    This \(a_\lambda\) is called the \textit{highest weight vector} 
    with respect to the ordered symplectic basis 
    \(a_1, \dots, a_g, b_1, \dots, b_g\). 
    For example, 
    \begin{equation}
        a_{[32]} = (a_1 \wedge a_2) \tensor (a_1 \wedge a_2) \tensor a_1 
        ,\quad 
        a_{[221]} = (a_1 \wedge a_2 \wedge a_3) \tensor (a_1 \wedge a_2).
    \end{equation}
\end{theorem}
\par Now we describe 
each \(\Symp\)-irreducible component of \(\wedge^\bullet \cgplus\) 
by a Young diagram. 
We simply write an \(\Symp\)-module isomorphic to \(V_\lambda\) 
as \(\lambda\) when we focus mainly on its isomorphism class. 
\begin{lemma}
    \begin{itemize}
        \item[(i)] If \(k > l \geq 1\), 
            \begin{equation*}
                \cg(k) \tensor \cg(l) \isom
                    \bigoplus_{0 \leq \lambda_2 \leq l+2} \ 
                    \bigoplus_{0 \leq \rho \leq l+2-\lambda_2}
                    \lbrack (k+l+4-\lambda_2-2\rho) \quad \lambda_2 \rbrack
            \end{equation*}
            In particular, if \(V \subset \cg(k) \tensor \cg(l)\) is 
            an \(\Symp\)-irreducible submodule corresponding to a Young diagram
            \(\lambda = \lbrack \lambda_1 \lambda_2\rbrack\), then the following conditions hold:
            \begin{align}
                l + 2 \geq \rho + \lambda_2, & &
                k + 2 \leq \rho + \lambda_1 & & 
                \text{where } \rho \defeq \frac{1}{2} (k+l+4-\lambda_1-\lambda_2) \in \ZInteger_{\geq 0} \label{possible young diagrams in cgk cgl}
            \end{align}
        \item[(ii)] If \(k \geq 1\), 
        \begin{equation*}
            \cg(k) \wedge \cg(k) \isom
                \bigoplus_{0 \leq \lambda_2 \leq k+2} \ 
                \bigoplus_{\substack{0 \leq \rho \leq k+2-\lambda_2 \\ \rho + \lambda_2 \text{ is odd}}}
                \lbrack (2k+4-\lambda_2-2\rho) \quad \lambda_2 \rbrack
        \end{equation*}
        In particular, if \(V \subset \cg(k) \tensor \cg(l)\) is 
        an \(\Symp\)-irreducible submodule corresponding to a Young diagram
        \(\lambda = \lbrack \lambda_1 \lambda_2\rbrack\), then the following conditions hold:
        \begin{align}
            \rho + \lambda_2 \text{ is odd}, & &
            \rho + \lambda_2 \leq k + 2 \leq \rho + \lambda_1 & & 
            \text{where } \rho \defeq \frac{1}{2} (2k+4-\lambda_1-\lambda_2) \in \ZInteger_{\geq 0} \label{possible young diagrams in cgk cgk}
        \end{align}
    \end{itemize}
\end{lemma}
\begin{proof}
(i) follows from the Littlewood-Richardson rule and branching rules.
\par (ii) follows from the same rules and the plethysm of symmetric polynomials 
\begin{equation*}
    e_2 \circ h_{k+2} = \sum_{1 \leq j \leq 2k+4-j, \ j \text{ is odd}} 
        s_{\lbrack (2k+4-j) \ j\rbrack}
\end{equation*}
where \(e_2\) is the second elementary symmetric function, 
\(h_{k+2}\) is the \((k+2)\)-nd complete symmetric function and 
\(s_{\lambda}\) is the Schur function corresponding to a Young diagram \(\lambda\)
(see \cite{MacDonald1995}).
\end{proof}
This lemma shows, for fixed \(k\) and \(l\), 
each Young diagram \(\lambda\) satisfying 
\eqref{possible young diagrams in cgk cgl} or \eqref{possible young diagrams in cgk cgk} 
corresponds to \textit{exactly one} \(\Symp\)-irreducible component in \(\cg(k) \wedge \cg(l)\), 
so that we identify such Young diagrams with them. 
From now, we denote simply by \(\lambda\) an \(\Symp\)-irreducible component 
\(V \subset \cg(k) \wedge \cg(l)\) isomorphic to \(V_\lambda\), like 
\(\lambda \subset \cg(k) \wedge \cg(l)\). 
\section{Detecting the highest weight vector}
\label{Detecting the highest weight vector}
\par We introduce some notations.
Define, for \(n,k \geq 2\) and \(1 \leq i,j,i_1, \dots, i_k \leq n\),
\begin{equation}
\begin{aligned}
\includegraphics{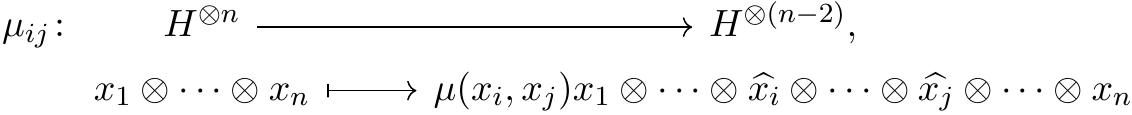}
\end{aligned}
\label{definition of contr}
\end{equation}
\begin{equation}
\includegraphics{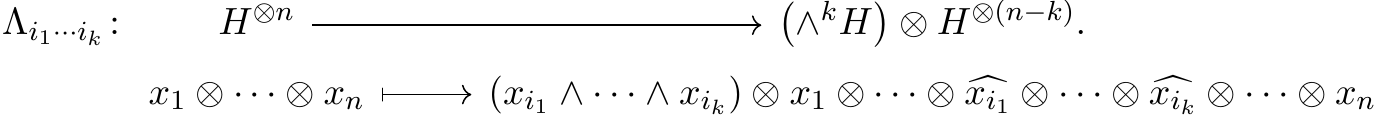}
\label{definition of alter}
\end{equation}
\par By regarding \(\cg(k)\) as an \(\Symp\)-submodule of \(H^{\tensor (k+2)}\), we have the following. 
\begin{lemma}
    Let \(k \geq l \geq 1\), \(f \in \cg(k)\) and \(h \in \cg(l)\). 
    \begin{itemize}
        \item[(1)] Let \(1 \leq i \leq k+2\) and \(1 \leq j \leq l+2\). 
        If \(k > l\), then 
        \(\contr{i}{j}(f \tensor h)=\contr{1,}{k+l+4}(f \tensor h)\) and 
        \(\alter{i}{j}(f \tensor h)=\alter{1,}{k+l+4}(f \tensor h)\). 
        If \(k = l\), then 
        \(\contr{i}{j}(f \wedge h)=\contr{1,}{2k+4}(f \wedge h)\) and 
        \(\alter{i}{j}(f \wedge h)=\alter{1,}{2k+4}(f \wedge h)\). 
        \item[(2)] For 
        \(f = x_1 \cdots x_{k+2}\), \(h = y_1 \cdots y_{l+2}\), 
        we have  
        \begin{equation}
            \begin{aligned}
                \contr{1,}{k+l+4}(f \wedge h) &{}= \sum_{\substack{1 \leq i \leq k+2 \\ 1 \leq j \leq l+2}} 
                    \mu(x_i, y_j) (x_1 \cdots \widehat{x_i} \cdots x_{k+2}) \symmprod (y_1 \cdots \widehat{y_j} \cdots y_{l+2}) \\
                \contr{1,}{k+l+4}(f \wedge h) &{}= \sum_{\substack{1 \leq i \leq k+2 \\ 1 \leq j \leq l+2}} 
                    (x_i \wedge y_j) \tensor (x_1 \cdots \widehat{x_i} \cdots x_{k+2}) \symmprod (y_1 \cdots \widehat{y_j} \cdots y_{l+2})
            \end{aligned}
        \end{equation}
        Moreover these formulae hold 
        if we switch \(\wedge\) and \(\symmprod\) each other. 
    \end{itemize}
\end{lemma}
\begin{proof}
    (1) holds because elements of \(S^{m+2} H (\subset H^{\tensor (m+2)}) \) 
    are invariant under a permutation of factors for any \(m \geq 1\). 
    \par (2) is shown by a direct computation. 
    Note that for all \(x_1, \dots, x_n \in H\), 
    \begin{equation}
        x_1 \cdots x_n = \sum_{1 \leq i \leq n} x_i \tensor (x_1 \cdots \widehat{x_i} \cdots x_n)
        \label{monomial in Htensorn}
    \end{equation}
    hold as an element of \(H^{\tensor n}\) by the definition of \(\iota\).
\end{proof}
We write \(\contr{1,}{n}\), \(\alter{1,}{n}\) on \(H^{\tensor n}\) as 
\(\contrend\), \(\alterend\) respectively. 
Based on this lemma, when we want to apply \(\contr{i}{j}\)s and \(\alter{i}{j}\)s 
to an element of \(\wedge^2 \cgplus\), 
it is enough to consider only \(\contrend\) and \(\alterend\). 
In \eqref{definition of alter}, for example, we can further apply \(\contr{i}{j}\) or 
\(\alter{j_1 \cdots j_l}{}\) \((1 \leq j_1, \dots, j_l \leq n-k)\) on the latter tensor component 
\(H^{\tensor (n-k)}\), again on \(H^{\tensor (n-k-2)}\) or \(H^{\tensor (n-k-l)}\), and so on. 
For \(n, p \geq 1\), \(k_1, \dots, k_p \geq 2\) and \(1 \leq i,j,i_1, \dots, i_l \leq n\), 
\begin{multline}
    \id_{(\wedge^{k_1} H) \tensor \cdots \tensor (\wedge^{k_p} H)} \tensor \contr{i}{j}
    \colon (\wedge^{k_1} H) \tensor \cdots \tensor (\wedge^{k_p} H) \tensor H^{\tensor n} \\
    \to (\wedge^{k_1} H) \tensor \cdots \tensor (\wedge^{k_p} H) \tensor H^{\tensor (n-2)} \label{abuse contr}
\end{multline}
is also denoted by the same symbol \(\contr{i}{j}\), and 
\begin{multline}
    \id_{(\wedge^{k_1} H) \tensor \cdots \tensor (\wedge^{k_p} H)} \tensor \alter{i_1 \cdots i_l}{}
    \colon (\wedge^{k_1} H) \tensor \cdots \tensor (\wedge^{k_p} H) \tensor H^{\tensor n} \\
    \to (\wedge^{k_1} H) \tensor \cdots \tensor (\wedge^{k_p} H) \tensor (\wedge^l H) \tensor H^{\tensor (n-l)} \label{abuse alter}
\end{multline}
by \(\alter{i_1 \cdots i_l}{}\) for short. 
Under this notation we can think of a composition of 
\(\contr{i}{j}\)s and \(\alter{i_1 \cdots i_l}{}\)s.
\par Next lemma gives a method to detect the highest weight vector 
of an \(\Symp\)-irreducible submodule of \(H^{\tensor n}\). 
For a completely reducible \(\Symp\)-module \(W\), 
its \(\Symp\)-irreducible submodule \(V\) and \(w \in W\), 
the image of \(w\) by the canonical projection \(W \twoheadrightarrow V\) 
is denoted by \(\restr{v}{V}\). 
\begin{lemma}\label{lemma detecting}
    Let \(\xi \in \cg(k_1) \wedge \cdots \wedge \cg(k_n)\) and let 
    \(V \subset \cg(k_1) \wedge \cdots \wedge \cg(k_n)\) be an \(\Symp\)-irreducible component, 
    which corresponds to a Young diagram \(\lambda\).
    Assume \(\xi\) is mapped to \(a_\lambda \in V_\lambda\) 
    by some compositions of \(\contr{i}{j}\)s and \(\alter{i_1 \cdots i_k}{}\)s. 
    Then \(\restr{\xi}{V}\) is nonzero and spans \(V\) as an \(\Symp\)-module.
\end{lemma}
\begin{proof}
    \(\contr{i}{j}\)s and \(\alter{i_1 \cdots i_k}{}\)s are 
    \(\Symp\)-equivariant homomorphisms 
    because \(\SymplTwogQ\) diagonally acts on 
    \(H^{\tensor n}\) and on its quotient modules.
    Since both \(V\) and \(V_\lambda\) are \(\Symp\)-irreducible, 
    if there is a nonzero \(\Symp\)-equivariant homomorphism 
    \(V \to V_\lambda\), then it is an isomorphism.
\end{proof}
In the proof of \thref{main theorem}, 
we frequently use the following.
\begin{lemma} \label{lemma detecting in cgplus2}
    Let \(k \geq l \geq 1\) and 
    \(\lambda = \left[ \lambda_1 \lambda_2 \right] \subset \cg(k) \wedge \cg(l)\).
    Set \(\rho \defeq \frac{1}{2} (k+l+4-\lambda_1-\lambda_2)\). 
    Then 
    \(\restr{\left( a_1^{k+2-\rho} a_3^\rho 
    \wedge a_1^{l+2-\lambda_2-\rho} a_2^{\lambda_2} b_3^\rho \right)}{\lambda}\)
    generates \(\lambda \subset \cg(k) \wedge \cg(l)\) as an \(\Symp\)-module.
\end{lemma}
\begin{proof}
We use \thref{lemma detecting}. 
If \(k > l\), then 
\begin{equation*}
    \begin{aligned}
        & a_1^{k+2-\rho} a_3^\rho \wedge a_1^{l+2-\lambda_2-\rho} a_2^{\lambda_2} b_3^\rho
        \quad \includegraphics{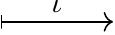} \quad 
        a_1^{k+2-\rho} a_3^\rho \tensor a_1^{l+2-\lambda_2-\rho} a_2^{\lambda_2} b_3^\rho \\
        \includegraphics{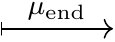} \quad  & 
        \rho^2 \cdot a_1^{k+2-\rho} a_3^{\rho-1} 
        \tensor a_1^{l+2-\lambda_2-\rho} a_2^{\lambda_2} b_3^{\rho-1} \\
        \vdots \phantom{1.5em} & \\
        \includegraphics{elemto_3em_contrend-crop.pdf} \quad & 
        \left( \rho ! \right)^2 \cdot a_1^{k+2-\rho} 
        \tensor a_1^{l+2-\lambda_2-\rho} a_2^{\lambda_2} \\
        \includegraphics{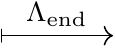} \quad & 
        \left( \rho ! \right)^2 (k+2-\rho) \lambda_2 \cdot 
        (a_1 \wedge a_2) \tensor a_1^{k+1-\rho} \tensor a_1^{l+2-\lambda_2-\rho} a_2^{\lambda_2-1} \\
        \vdots \phantom{1.5em} & \\
        \includegraphics{elemto_3em_alterend-crop.pdf} \quad & 
        \left( \rho ! \right)^2 
        \frac{(k+2-\rho)!}{(k+2-\lambda_2-\rho)!} \lambda_2 ! \cdot 
        (a_1 \wedge a_2)^{\tensor \lambda_2} \tensor a_1^{k+2-\lambda_2-\rho} \tensor a_1^{l+2-\lambda_2-\rho}  \\
        ={} & 
        \left( \rho ! \right)^2 
        (k+2-\rho)! (l+2-\lambda_2-\rho)! \lambda_2 ! \cdot 
        (a_1 \wedge a_2)^{\tensor \lambda_2} \tensor a_1^{\tensor \lambda_1 - \lambda_2},
    \end{aligned}
\end{equation*}
where \(\contrend\) is applied \(\rho\) times and \(\alterend\) is applied \(\lambda_2\) times. 
We used \eqref{monomial in Htensorn} to find coefficients in each \(\contrend\) or \(\alterend\). 
In the last equality, we used the fact that \(x^m = m ! x^{\tensor m} \in H^{\tensor m}\) 
for any \(m \geq 1\) and any \(x \in H\). 
This is shown by the definition of \(\iota\). 
Since \((a_1 \wedge a_2)^{\tensor \lambda_2} \tensor a_1^{\tensor \lambda_1 - \lambda_2} = a_\lambda\), 
the statement follows.
\par If \(k = l\), then by a similar procedure, 
\begin{equation*}
    \begin{aligned}
        & a_1^{k+2-\rho} a_3^\rho \wedge a_1^{l+2-\lambda_2-\rho} a_2^{\lambda_2} b_3^\rho \\
        \includegraphics{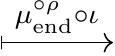} \quad & 
        \left\{
        \begin{array}{lc}
            \left( \rho ! \right)^2 \cdot a_1^{k+2-\rho} 
            \wedge a_1^{l+2-\lambda_2-\rho} a_2^{\lambda_2} 
            & (\text{if } \rho \text{ is even}) \\[1.5ex]
            \left( \rho ! \right)^2 \cdot a_1^{k+2-\rho} 
            \symmprod a_1^{l+2-\lambda_2-\rho} a_2^{\lambda_2} 
            & (\text{if } \rho \text{ is odd}) \\
        \end{array} \right. \\
        \includegraphics{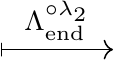} \quad & 
        \left( \rho ! \right)^2 
        \frac{(k+2-\rho)!}{(k+2-\lambda_2-\rho)!} \lambda_2 ! \cdot 
        (a_1 \wedge a_2)^{\tensor \lambda_2} \tensor (a_1^{k+2-\lambda_2-\rho} \symmprod a_1^{l+2-\lambda_2-\rho} ) \\
        ={} & 
        2 \left( \rho ! \right)^2 
        (k+2-\rho)! (l+2-\lambda_2-\rho)! \lambda_2 ! \cdot 
        (a_1 \wedge a_2)^{\tensor \lambda_2} \tensor a_1^{\tensor \lambda_1 - \lambda_2} 
        .
    \end{aligned}
\end{equation*}
Here \(\contrend^{\circ \rho}\) is the \(\rho\)-time compositions of \(\contrend\), and 
\(\alterend^{\circ \lambda_2}\) is the \(\lambda_2\)-time compositions of \(\alterend\). 
Note that \(\rho + \lambda_2\) is always odd by \eqref{possible young diagrams in cgk cgk} 
and that 
\(x^{\tensor l} \symmprod x^{\tensor m} = 2 x^{\tensor (l+m)} \in H^{\tensor (l+m)}\) 
holds for \(l,m \geq 1\) and \(x \in H\) from the definition of \(\symmprod\). 
\end{proof}
\section{Proof of the main theorem} 
\label{Proof of main theorem}
Fix a weight \(w \geq 4\), 
\(k \geq l \geq 1\) such that \(k + l = w\),
and an \(\Symp\)-irreducible submodule \(V\) corresponding to a Young diagram
\(\lambda = [\lambda_1 \lambda_2] \subset \cg(k) \wedge \cg(l)\). 
Set 
\(\rho \defeq \frac{1}{2} (k + l + 4 - \lambda_1 - \lambda_2)\).
\\
\par Here is our strategy of the proof. 
First, since \(\cg(w) = S^{w+2} H\) itself is \(\Symp\)-irreducible, we have 
\(\partial_2 (V) = 0\) if \(\lambda \neq \lbrack w+2 \rbrack\). 
Therefore we focus respectively on the case \(\lambda \neq \lbrack w+2 \rbrack\) 
and the case \(\lambda = \lbrack w+2 \rbrack\). 
\par Next we take an element \(\omega_3 \in \wedge^3 \cgplus\). 
We denote by \(\restr{\omega_3}{\lambda}\) 
the image of \(\omega_3\) by the projection to 
the isotypical component of \(\wedge^3 \cgplus\) corresponding to \(\lambda\). 
For example, if \(\partial_3 (\omega_3) \in \cg(k) \wedge \cg(l)\) and it is mapped by 
\(\contrend\)s and \(\alterend\)s to a nonzero constant multiple of 
\(a_\lambda \in \left(\wedge^2 H \right)^{\tensor \lambda_2} \tensor H^{\tensor (\lambda_1-\lambda_2)}\), 
then \(\restr{\partial_3 (\omega_3)}{V}\) generates \(V\) as an \(\Symp\)-module. 
Since \(\partial_3\) is \(\Symp\)-equivariant, we have 
\(\partial_3(\restr{\omega_3}{\lambda}) = \restr{\partial_3(\omega_3)}{V}\). 
In this case, \(\restr{\partial_3(\omega_3)}{V}\), 
which is a generator of \(V\), is in \(\Image (\partial_3)\). 
Consequently whole \(V\) is in \(\Image (\partial_3)\). 
This can happen if \(\lambda \neq \lbrack w+2 \rbrack\). 
For \(\lambda = \lbrack w+2 \rbrack\), 
we have to determine \(\Kernel \partial_2\) restricted to 
the isotypical component of \(\wedge^2 \cgplus\) 
corresponding to \(\lambda = \lbrack w+2 \rbrack\), 
and show that its generators are in \(\Image (\partial_3)\) by a similar argument. 
\\
\par We divide our argument for the following cases: 
\begin{itemize}
    \item [(I)] the case \(\rho = 0\),
    \item [(II)] the case \(\rho = 1\) and \ 
        \begin{enumerate*}[itemjoin={\qquad}]
            \item [(i)] \(\lambda_2 \geq 1\),
            \item [(ii)] \(\lambda_2 = 0\),
        \end{enumerate*}
    \item [(III)] the case \(\rho \geq 2\) and \ 
    \begin{enumerate*}[itemjoin={\qquad}]
        \item [(i)] \(k-\rho \geq 1\),
        \item [(ii)] \(k-\rho = 0\),
        \item [(iii)] \(k-\rho \leq -1\).
    \end{enumerate*}
\end{itemize}
These cases clearly cover all possible patterns.
\subsection{The case (I)} 
\label{The case (I)}
Suppose \(\rho = 0\). We set 
\begin{align}
    \omega_3 \defeq a_1^k a_4 \wedge a_1^2 b_4 \wedge a_1^{\lambda_1 -k-2} a_2^{\lambda_2}
        \ \in \cg(k-1) \wedge \cg(1) \wedge \cg(l)
        .
\end{align}
\par Since 
\(\partial_3(\omega_3) = a_1^{k+2} \wedge a_1^{\lambda_1-k-2} a_2^{\lambda_2}\), 
by using \thref{lemma detecting in cgplus2}, 
the element \(\restr{\partial_3(\omega_3)}{\lambda} = \partial_3(\restr{\omega_3}{\lambda})\) 
generates \(\lambda \in \cg(k) \wedge \cg(l)\). 
\subsection{The case (II)(i)} 
\label{The case (II)(i)}
%
%
\begin{lemma} \label{lemma case II i computation}
    Let \(p,q,r \geq 0\) with \(p \geq r\).
    \begin{itemize}
        \item[(1)] \(\alterend^{\circ (r+1)} (a_1^p a_2 \tensor a_1^q a_2^r) 
        = -p! q! (r+1)! (a_1 \wedge a_2)^{\tensor (r+1)} \tensor a_1^{p+q-r-1}\).
        \item[(2)] 
            \(\alterend^{\circ (r+1)} (a_1^p a_2 \symmprod a_1^q a_2^r) 
        = -2 \cdot p! q! (r+1)! (a_1 \wedge a_2)^{\tensor (r+1)} \tensor a_1^{p+q-r-1}\)
        if \(r\) is even. 
    \end{itemize}
\end{lemma}
\begin{proof}
Easily shown by induction on \(r\).
\end{proof}
%
Suppose \(\rho = 1\) and \(\lambda_2 \geq 1\). 
We set 
\begin{equation}
    \begin{aligned}
        \omega_3 \defeq{} & a_1^2 a_4 \wedge a_1^{k-1} a_3 b_4 \wedge a_1^{l+1-\lambda_2} a_2^{\lambda_2} b_3
            - a_1^2 a_4 \wedge a_1^{k-2} a_2 a_3 b_4 \wedge a_1^{l+2-\lambda_2} a_2^{\lambda_2-1} b_3 \\
            & \in{} \cg(1) \wedge \cg(k-1) \wedge \cg(l)
            .
    \end{aligned}
\end{equation}
Note that all exponents here are certainly nonnegative because 
\(l+1-\lambda_2=l+2-(\lambda_2+\rho) \geq 0\).
\par If \(k > l\), then 
\begin{align}
    \omega_3 \quad \includegraphics{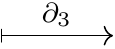} \quad & 
        a_1^{k+1} a_3 \wedge a_1^{l+1-\lambda_2} a_2^{\lambda_2} b_3
        - a_1^k a_2 a_3 \wedge a_1^{l+2-\lambda_2} a_2^{\lambda_2-1} b_3 
        \ (\in \cg(k) \wedge \cg(l)) \nonumber \\
    \includegraphics{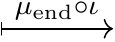} \quad & 
        a_1^{k+1} \tensor a_1^{l+1-\lambda_2} a_2^{\lambda_2}
       - a_1^k a_2 \tensor a_1^{l+2-\lambda_2} a_2^{\lambda_2-1} \nonumber \\
    \includegraphics{elemto_3em_alterendcirclambda2-crop.pdf} \quad &
        ((k+1)! (l+1-\lambda_2)! \lambda_2! + k! (l+2-\lambda_2)! \lambda_2 !) 
        (a_1 \wedge a_2)^{\tensor \lambda_2} \tensor a_1^{\tensor \lambda_1-\lambda_2} 
        \label{case II i k nleq l}
\end{align}
\par and if \(k=l\), then 
\begin{align}
    \omega_3 \quad \includegraphics{elemto_3em_partial3-crop.pdf} \quad & 
        a_1^{k+1} a_3 \wedge a_1^{l+1-\lambda_2} a_2^{\lambda_2} b_3
        - a_1^k a_2 a_3 \wedge a_1^{l+2-\lambda_2} a_2^{\lambda_2-1} b_3 
        \ (\in \cg(k) \wedge \cg(l)) \nonumber \\
    \includegraphics{elemto_3em_contrendcirciota-crop.pdf} \quad & 
        a_1^{k+1} \symmprod a_1^{l+1-\lambda_2} a_2^{\lambda_2}
        - a_1^k a_2 \symmprod a_1^{l+2-\lambda_2} a_2^{\lambda_2-1} \nonumber \\
    \includegraphics{elemto_3em_alterendcirclambda2-crop.pdf} \quad &
        2((k+1)! (l+1-\lambda_2)! \lambda_2! + k! (l+2-\lambda_2)! \lambda_2 !) 
        (a_1 \wedge a_2)^{\tensor \lambda_2} \tensor a_1^{\tensor \lambda_1-\lambda_2} 
        .
        \label{case II i k equals l}
\end{align}
We used \thref{lemma case II i computation} and the proof of \thref{lemma detecting in cgplus2}  
at the last rows in \eqref{case II i k nleq l} and \eqref{case II i k equals l}. 
%
%
\subsection{The case (II)(ii)} 
\label{The case (II)(ii)}
Suppose \(\rho = 1\) and \(\lambda_2 = 0\). 
In this case, we have \(\lambda = [w+2] \isom \wedge^1 \cg(w)\), 
so that we have to see \(\Kernel(\partial_2)\) in detail. 
Only here we do not fix \(k\) and \(l\). 
\(k\) runs over \(\lfloor \frac{w}{2} \rfloor \leq k \leq w - 1\) because 
\(k \geq l \geq 1\) and \(k + l = w\). 
\par For each such \(k\), the space \(\cg(k) \wedge \cg(l)\) contains 
an \(\Symp\)-irreducible component \(\lambda = [w+2]\), 
which is generated by 
\(v^{(k)} \defeq a_1^{k+1} a_2 \wedge a_1^{l+1} b_2 \in \cg(k) \wedge \cg(l)\) 
by the similar argument to \thref{lemma detecting}. 
Let \(\pi^{(k)} \colon \cg(k) \wedge \cg(l) \twoheadrightarrow [w+2]\) 
be the projection to the component \([w+2]\) and 
set \(u^{(k)} \defeq \pi^{(k)} (v^{(k)}) \in [w+2] \subset \cg(k) \wedge \cg(l)\).
We easily see that \(\partial_2 (v^{(k)}) = a_1^{k+l+2} = a_1^{w+2}\) for any \(k\). 
Therefore 
\(K \defeq \restr{\Kernel(\partial_2)}{\lbrack w+2 \rbrack }\) is generated by 
\(\setin{u^{(k)} - u^{(m)}}{\lfloor \frac{w}{2} \rfloor \leq k < m \leq w - 1}\) 
as an \(\Symp\)-module. 
Then it is enough to show the following:
\begin{lemma}
\(\setin{v^{(k)} - v^{(m)}}{\lfloor \frac{w}{2} \rfloor \leq k < m \leq w - 1}
    \subset \Image(\partial_3)\) 
\end{lemma}
\begin{proof}
Let \(\lfloor \frac{w}{2} \rfloor \leq k < m \leq w - 1\). 
We set 
\begin{equation}
    \begin{aligned}
        \omega_3 \defeq -{}  a_1^{k+1} a_2 \wedge a_1^{m-k} a_2 b_2 \wedge a_1^{w-m+1} b_2 
            +{}  a_1^{k+1} a_3 \wedge a_1^{m-k} a_2 b_2 \wedge a_1^{w-m+1} b_3
    \end{aligned}
\end{equation}
then obtain 
\begin{align}
    \partial_3 (\omega_3) ={} & 
        - \left( a_1^{m+1} a_2 \wedge a_1^{w-m+1} b_2 
        - a_1^{w+2-m+k} \wedge a_1^{m-k} a_2 b_2
        - a_1^{k+1} a_2 \wedge a_1^{w-k+1} b_2 \right) \nonumber \\
        & - a_1^{w+2-m+k} \wedge a_1^{m-k} a_2 b_2 \nonumber \\
        ={} & v^{(k)} - v^{(m)}.
\end{align}
Hence \(v^{(k)} - v^{(m)} \in \Image(\partial_3)\).
\end{proof}
\par We also see that \(K \subset \Image(\partial_3)\) by 
restricting each \(\omega_3\) in the proof above to 
the isotypical component of \(\wedge^3 \cgplus\) corresponding to \([w+2]\).
\subsection{The case (III)(i)} 
\label{The case (III)(i)}
Suppose \(\rho \geq 2\) and \(k-\rho \geq 1\). 
We set 
\begin{equation}
    \begin{aligned}
        \omega_3 \defeq \phantom{-{}}& (\rho-1) \cdot 
            a_1^2 a_4 \wedge a_1^{k-\rho} a_3^\rho b_4 \wedge a_1^{l+2-\lambda_2-\rho} a_2^{\lambda_2} b_3^\rho \nonumber \\
            -{} & \rho \cdot 
            a_1^2 a_4 \wedge a_1^{k-\rho-1} a_3^\rho b_3 b_4 \wedge a_1^{l+3-\lambda_2-\rho} a_2^{\lambda_2} b_3^{\rho-1} 
            \ \in{} \cg(1) \wedge \cg(k-1) \wedge \cg(l)
            .
    \end{aligned}
\end{equation}
All exponents here are certainly nonnegative by the assumption here.
Then 
\begin{align*}
    \partial_3 (\omega_3) ={} & 
    (\rho - 1) \cdot a_1^{k-\rho+2} a_3^\rho \wedge a_1^{l+2-\lambda_2-\rho} a_2^{\lambda_2} b_3^\rho \nonumber \\
    &{} - \rho \cdot a_1^{k-\rho+1} a_3^\rho b_3 \wedge a_1^{l+3-\lambda_2-\rho} a_2^{\lambda_2} b_3^{\rho-1} 
    \ \in \cg(k) \wedge \cg(l) 
    .
\end{align*}
The second term of \(\partial_3(\omega_3)\) vanishes in the \(\rho\)-th contraction, 
while the first term survives by \thref{lemma detecting in cgplus2}.
\subsection{The case (III)(ii)} 
\label{The case (III)(ii)}
Suppose \(k = \rho \geq 2\). 
This case is further divided into the following:
\begin{itemize}
    \item [(a)] the case \(k \geq 4\),
    \item [(b)] the case \(k = 3\) and \ 
        \begin{enumerate*}[itemjoin={\qquad}]
            \item [(1)] \(\lambda_2 \geq 1\),
            \item [(2)] \(\lambda_2 = 0\),
        \end{enumerate*}
    \item [(c)] the case \(k=2\).
\end{itemize}
In the case (a), we set 
\begin{align}
    \omega_3 & \defeq 
        a_1^2 a_4 \wedge a_3^k b_4 \wedge a_1^{l+2-\lambda_2-k} a_2^{\lambda_2} b_3^k \nonumber \\
    & \includegraphics{elemto_3em_partial3-crop.pdf} \quad 
        a_1^2 a_3^k \wedge a_1^{l+2-\lambda_2-k} a_2^{\lambda_2} b_3^k 
        - k^2 \cdot a_1^2 a_4 \wedge a_1^{l+2-\lambda_2-k} a_2^{\lambda_2} a_3^{k-1} b_3^{k-1} b_4 
        .
    \label{image (III)(i)(a)}
\end{align}
The second term of \eqref{image (III)(i)(a)} is an element of 
\(\cg(1) \wedge \cg(k+l-1)\), which does not contain an \(\Symp\)-irreducible component 
corresponding to \(\lambda\) because \(1 + 2 \ngeq \rho + \lambda_2 \geq \rho = k \geq 4\) 
violates the condition \eqref{possible young diagrams in cgk cgl}. 
On the other hand, the first term of \eqref{image (III)(i)(a)} is mapped to 
\begin{equation}
    \left( \rho !\right)^2 \lambda_2 ! (k-\rho+2) ! (l+2-\lambda_2-\rho) ! \cdot 
    (a_1 \wedge a_2)^{\tensor \lambda_2} \tensor a_1^{\tensor \lambda_1-\lambda_2}
\end{equation}
by \(\alterend^{\circ \lambda_2} \circ \contrend^{\circ \rho} \circ \iota\). 
Therefore \(\partial_3 (\restr{\omega_3}{\lambda})\) generates 
\(\lambda \in \cg(k) \wedge \cg(l)\). 
\par In the case (b), \(\omega_3\) to be defined and its image by \(\partial_3\) 
are as below.
\begin{center}
\begin{tabular}{c|ccc}
    & \multirow{2}{*}{(1) \(\lambda_2 \geq 1\)} &
    \multicolumn{2}{c}{(2) \(\lambda_2 = 0\)} \\
    & & \(l=1\) & \(l=2,3\) \\ \hline
    \(\omega_3\) & 
    \(a_1^2 a_4 \wedge a_3^3 b_4 \wedge a_1^{l-\lambda_2-1} a_2^{\lambda_2} b_3^3\) & 
    \(a_1^2 a_4 \wedge a_3^3 b_4 \wedge b_3^3 \) & 
    \(\begin{array}{c} 
        2 a_1^2 a_3 \wedge a_3^3 b_4 \wedge a_1^{l-1} b_3^3 \\
        -3 a_1^2 a_4 \wedge a_1 a_3^2 b_4 \wedge a_1 a_3 b_3^3
    \end{array}\) \\
    \(\partial_3(\omega_3)\) & 
    \( \begin{array}{c}
        a_1^2 a_3^3 \wedge a_1^{l-\lambda_2-1} a_2^{\lambda_2} b_3^3 \\
        - 9 a_1^2 a_4 \wedge a_1^{l-\lambda_2-1} a_2^{\lambda_2} a_3^2 b_3^2
    \end{array}\) & 
    \(\begin{array}{c} 
        a_1^2 a_3^3 \wedge b_3^3 \\ - 9 a_1^2 a_4 \wedge a_3^2 b_3^2 
    \end{array}\) & 
    \(\begin{array}{c} 
        2 a_1^2 a_3^3 \wedge a_1^{l-1} b_3^3 \\
        -3 a_1^3 a_3^2 \wedge a_1^l a_3 b_3^3
    \end{array}\)
\end{tabular}
\end{center}
\vspace{11pt}
The second term of \(\partial_3(\omega_3)\) in the case (1) is an element of 
\(\cg(1) \wedge \cg(k+l-1)\), which does not contain 
\(\lambda\) as its \(\Symp\)-irreducible component because \(\rho \geq 4\). 
Moreover the second terms of \(\partial_3(\omega_3)\) in the both cases in (2) 
vanish when they are mapped by \(\contrend^{\circ 3} \circ \iota\). 
Since all the first terms in the cases (1) and (2) are mapped to 
nonzero constant multiples of 
\((a_1 \wedge a_2)^{\lambda_2} \tensor a_1^{\lambda_1-\lambda_2}\) 
by \(\alterend^{\circ \lambda_2} \circ \contrend^{\circ 3} \circ \iota\), 
the element \(\partial_3(\restr{\omega_3}{\lambda})\) in each case generates
\(\lambda \subset \cg(3) \wedge \cg(l)\).
\par Finally we consider the case (c). 
Since \(k \geq l \geq 1\) and \(k + l = w \geq 4\), we get \(\rho = k = l = 2\), 
so that \(\lambda_1 + \lambda_2 = 4\). 
By the condition \eqref{possible young diagrams in cgk cgk}, 
the only possible \(\lambda\) is \(\lambda_1 = 3\) and \(\lambda_2 = 1\) 
because \(\lambda_2\) must be an odd number. 
Then 
\begin{equation}
    \begin{aligned}
        \omega_3 & \defeq 
            a_1^2 a_4 \wedge a_3^2 b_4 \wedge a_1 a_2 b_3^2 
            - 2 a_1^2 a_4 \wedge a_1 a_3 b_4 \wedge a_2 a_3 b_3^2 \nonumber \\
        & \includegraphics{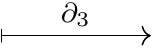} \quad
            a_1^2 a_3^2 \wedge a_1 a_2 b_3^2
            -2 a_1^3 a_3 \wedge a_2 a_3 b_3^2 \nonumber \\
        & \includegraphics{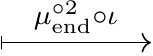} \quad
            4 a_1^2 \wedge a_1 a_2 
        \quad \includegraphics{elemto_3em_alterend-crop.pdf} \quad 
            8 (a_1 \wedge a_2) \tensor (a_1 \symmprod a_1)
            = 16 (a_1 \wedge a_2) \tensor a_1^{\tensor 2}
        .  
    \end{aligned}
\end{equation}
\subsection{The case (III)(iii)} 
\label{The case (III)(iii)}
Suppose \(k-\rho \leq -1\).
We argue in two cases:
\begin{enumerate*}
    \item [(a)] the case \(k \geq 3\), 
    \item [(b)] the case \(k = 2\).
\end{enumerate*}
\par In the case (a), we get 
\begin{align}
    k \leq \rho - 1 \leq \rho + \lambda_2 - 1 \leq l + 2 - 1 = l + 1.
\end{align}
Since \(k \geq l \geq 1\), we see all the possible patterns are 
\begin{center}
\begin{tabular}{lcl}
    (1) \(k = l + 1 = \rho - 1\) and \(\lambda = [1]\),
    & \phantom{MMM} & (2) \(k = l = \rho - 2\) and \(\lambda = [0]\), \\
    (3) \(k = l = \rho - 1\) and \(\lambda = [2]\),
    & & (4) \(k = l = \rho - 1\) and \(\lambda = [11]\).
\end{tabular}
\end{center}
Note that \(\rho \geq k + 1 \geq 4\) in all cases.
\par \(\omega_3\) to be defined and its image by \(\partial_3\) 
are as below:
\vspace{11pt}
\begin{center}
    \begin{tabular}{c|cc}
        & Case (1) &
        Case (2) \\ \hline
        \(\omega_3\) & 
        \(a_1 a_3^{k+1} \wedge a_4 b_3^{k-1} \wedge b_3^2 b_4\) &
        \(a_1^2 a_3 \wedge a_1^k b_3 \wedge b_1^{k+2}\) \\[2ex]
        \(\partial_3(\omega_3)\) & 
        \( \begin{array}{l}
            - a_1 a_3^{k+1} \wedge b_3^{k+1} \\
            + (\blacktriangle \in \cg(2k-2) \wedge \cg(1)) \\
            + (\blacktriangle \in \cg(k+1) \wedge \cg(k-2)) \\
        \end{array}\) & 
        \(\begin{array}{l} 
            \phantom{{}-{}} a_1^{k+2}\wedge b_1^{k+2} \\
            + (\blacktriangle \in \cg(k-1) \wedge \cg(k+1)) \\
            + (\blacktriangle \in \cg(1) \wedge \cg(2k-1))
        \end{array}\)
    \end{tabular}
\end{center}
\vspace{11pt}
\begin{center}
    \begin{tabular}{c|cc}
        & Case (3) &
        Case (4) \\ \hline
        \(\omega_3\) & 
        \(\begin{array}{l}
            \phantom{{}-{}} a_3^2 a_4 \wedge a_1 a_3^{k-1} b_4 \wedge a_1 b_3^{k+1} \\
            - 2 a_1 a_3 a_4 \wedge a_1 a_3^{k-1} b_4 \wedge a_3 b_3^{k+1}
        \end{array}\) &
        \(a_3^2 a_4 \wedge a_1 a_3^{k-1} b_4 \wedge a_2 b_3^{k+1}\) \\[4ex]
        \(\partial_3(\omega_3)\) & 
        \( \begin{array}{l}
            \phantom{{}-{}} a_1 a_3^{k+1} \wedge a_1 b_3^{k+1} \\
            - a_1^2 a_3^k \wedge a_3 b_3^{k+1} \\
            + (\blacktriangle \in \cg(1) \wedge \cg(2k-1)) \\
        \end{array}\) & 
        \(\begin{array}{l} 
            \phantom{{}-{}} a_1 a_3^{k+1} \wedge a_2 b_3^{k+1} \\
            + (\blacktriangle \in \cg(1) \wedge \cg(2k-1)) \\
            + (\blacktriangle \in \cg(k-1) \wedge \cg(k+1)) 
        \end{array}\)
    \end{tabular}
\end{center}
\vspace{11pt}
Here \(\blacktriangle\) means some element which we need not to specify.
\par In all cases, 
both \(\cg(1) \wedge \cg(2k-1)\) and \(\cg(1) \wedge \cg(2k-2)\) 
do not contain \(\lambda\) as an \(\Symp\)-irreducible component 
since \(\rho \geq 4\). 
In the case (1), \(\rho = k+1\) implies 
\(\cg(k+1) \wedge \cg(k-2)\) has no \(\lambda\) component. 
In the case (2), \(\rho = k+2 \nleq (k-1)+2\) implies 
\(\cg(k-1) \wedge \cg(k+1)\) has no \(\lambda\) component. 
In the case (4), \(\rho + 1 = k+2 \nleq (k-1)+2\) implies 
\(\cg(k-1) \wedge \cg(k+1)\) has no \(\lambda\) component. 
\par In each case of (1), (2) and (4), the remaining term which is specifically written 
is clearly mapped to a nonzero constant multiple of \(a_\lambda\) by 
\(\alterend^{\circ \lambda_2} \circ \contrend^{\circ \rho} \circ \iota\). 
This holds even for the case (3) 
because it is true for the first term of \(\partial_3(\omega_3)\), 
and because the second term of \(\partial_3(\omega_3)\) vanishes 
by the \((k+1)\)-st contraction.
\par From the above, 
\(\partial_3(\restr{\omega_3}{\lambda})\) in each case generates 
\(\lambda \subset \cg(k) \wedge \cg(l)\).
\\
\par In the case (b), we get \(k=l=2\) and \(\rho=3\). 
Since \(\rho + \lambda_2\) must be an odd number, 
we also see that \(\lambda_1 = 2\) and \(\lambda_2 = 0\). 
We adopt a slightly different approach from the others. 
Set and see
\begin{equation}
    \begin{aligned}
        \zeta_3 \defeq &  
            a_1 a_3^3 \wedge a_1 a_4 b_3 \wedge b_3^2 b_4
            - a_3^3 a_4 \wedge b_3^3 \wedge a_1^2 b_4 
            \in{} \cg(2) \wedge \cg(1) \wedge \cg(1) \\
        \includegraphics{elemto_3em_partial3-crop.pdf} \quad & 
        3 a_1^2 a_3^2 a_4 \wedge b_3^2 b_4
        - 6 a_1 a_3^2 b_3 b_4 \wedge a_1 a_4 b_3 
        - 9 a_3^2 a_4 b_3^2 \wedge a_1^2 b_4 \\ 
        &- a_1 a_3^3 \wedge a_1 b_3^3 + a_1^2 a_3^3 \wedge b_3^3
        .
    \end{aligned}
\end{equation}
Set  
\begin{equation}
    \begin{aligned}
        \eta &{}\defeq 3 a_1^2 a_3^2 a_4 \wedge b_3^2 b_4
            - 6 a_1 a_3^2 b_3 b_4 \wedge a_1 a_4 b_3 
            - 9 a_3^2 a_4 b_3^2 \wedge a_1^2 b_4 
            + \restr{(a_1^2 a_3^3 \wedge b_3^3)}{\lambda} \\
        & \in \cg(3) \wedge \cg(1) 
        .
    \end{aligned}
\end{equation}
\(\partial_2(\eta) = 0\) is shown by a direct computation. 
Therefore there exists \(\zeta_3^\prime \in \wedge^3 \cgplus\) 
such that \(\partial_3(\zeta_3^\prime) = \eta\) 
from the case \(k=3\) and \(l=1\), which is done in the argument so far. 
Hence 
\begin{equation}
    \partial_3 (\zeta_3 - \zeta_3^\prime) 
    = - a_1 a_3^3 \wedge a_1 b_3^3
    + a_1^2 a_3^3 \wedge b_3^3
    - \restr{(a_1^2 a_3^3 \wedge b_3^3)}{\lambda} . 
\end{equation}
Since \(\restr{\left( a_1^2 a_3^3 \wedge b_3^3
- \restr{(a_1^2 a_3^3 \wedge b_3^3)}{\lambda} \right)}{\lambda} = 0\), 
the element \(\partial_3 (\restr{(\zeta_3 - \zeta_3^\prime)}{\lambda})
= - \restr{(a_1 a_3^3 \wedge a_1 b_3^3)}{\lambda}\) 
generates \(\lambda \subset \cg(k) \wedge \cg(l)\).
\section{Lower weight cases}
\label{Lower weight cases}

For \(1 \leq w \leq 3\), the weight \(w\) part \(H_2 (\cgplus)_w\) is 
determined as follows. 
\begin{lemma} \label{lemma lower weight cases}
    If \(g \geq 4\), then 
    \( H_2 (\cgplus)_1 = 0\), 
    \( H_2 (\cgplus)_2 = [51] + [33] + [22] + [11] + [0]\), and 
    \( H_2 (\cgplus)_3 = [1]\).
\end{lemma}
\begin{proof}
\( H_2 (\cgplus)_1 = 0\) is obvious because no \(k \geq l \geq 1\) satisfy \(k+l=1\).
\par Since the weight 2 part of \({\wedge}^3 \cgplus\) is zero and 
since \(\partial_2 = [,] \colon {\wedge}^2 \cg(1) \to \cg(2) = S^4 H = [4] \) is surjective, 
we have \(H_2 (\cgplus)_2 = {\wedge}^2 \cg(1) / \cg(2) \). 
The \(\Symp\)-irreducible decomposition of \({\wedge}^2 \cg(1)\) is 
\([51] + [33] + [4] + [22] + [11] + [0]\), therefore the statement follows.
\par The \(\Symp\)-irreducible decomposition of \(\cg(2) \tensor \cg(1)\) is 
\begin{equation}
    \cg(2) \tensor \cg(1) = [7] + [61] + [52] + [43] + [5] + [41] + [32] + [3] + [21] + [1]. 
\end{equation}
The space \({\wedge}^3 \cg(1)\) does not have \([5]\) and \([1]\) as its \(\Symp\)-irreducible components.
We use the same method as in the case \(w \geq 4\) about all the other 
\(\Symp\)-irreducible components.
It is enough to define \(\omega_3\) as the following. 
\begin{center}
\begin{tabular}{l}
    \begin{tabular}{c|cccc}
        & [7] & [61] & [52] & [43] \\ \hline
        \(\omega_3\) & 
        \(a_1^2 a_4 \wedge a_1^2 b_4 \wedge a_1^3\) & 
        \(a_1^2 a_4 \wedge a_1^2 b_4 \wedge a_1^2 a_2\) & 
        \(a_1^2 a_4 \wedge a_1^2 b_4 \wedge a_1 a_2^2\) & 
        \(a_1^2 a_4 \wedge a_1^2 b_4 \wedge a_2^3\)
    \end{tabular}
    \\[4ex]
    \begin{tabular}{c|cc}
        & [41] & [32] \\ \hline
        \(\omega_3\) & 
        \(\begin{array}{l}
            a_1^2 a_4 \wedge a_1 a_3 b_4 \wedge a_1 a_2 b_3 \\
            \phantom{MMM} - a_1^2 a_4 \wedge a_2 a_3 b_4 \wedge a_1^2 b_3
        \end{array}\) & 
        \(\begin{array}{l}
            a_1^2 a_4 \wedge a_1 a_3 b_4 \wedge a_2^2 b_3 \\
            \phantom{MMM} - a_1^2 a_4 \wedge a_2 a_3 b_4 \wedge a_1 a_2 b_3
        \end{array}\)
    \end{tabular}
    \\[6ex]
    \begin{tabular}{c|cc}
        & [3] & [21] \\ \hline
        \(\omega_3\) & 
        \(a_1^2 a_4 \wedge a_3^2 b_4 \wedge a_1 b_3^2\) & 
        \(a_1^2 a_4 \wedge a_3^2 b_4 \wedge a_2 b_3^2\) 
    \end{tabular}
\end{tabular}
\end{center}
\vspace{11pt}
Again, since \(\partial_2 = [,] \colon \cg(2) \wedge \cg(1) \to \cg(3) = S^5 H = [5] \) is surjective, 
we have \(H_2 (\cgplus)_3 = (\cg(2) \tensor \cg(1)) / 
\left( 
    \cg(3) \oplus \Image(\partial_3 \colon \wedge^3 \cg(1) \to \cg(2) \wedge \cg(1))
\right) = [1]\).
\end{proof}
From \thref{main theorem} and \thref{lemma lower weight cases}, we obtain the following. 
\begin{corollary}
    \(H_2 (\cgplus) = [51] + [33] + [22] + [11] + [1] + [0]\) if \(g \geq 4\).
\end{corollary}
Moreover, we know the \(\Symp\)-irreducible decomposition of \(\wedge^3 \cg(1)\): 
\begin{equation}
    \begin{aligned}
        \wedge^3 \cg(1) ={}&
            [711] + [63] + [531] + [333] + [7] + [61] + 2 [52] + [43] + [421] + [322]\\
            &+ 2 [41] + [32] + 2 [311] + 3 [3] + [21].
    \end{aligned}
\end{equation}
We have already seen that 
\begin{equation}
    \Image \left( \restr{ \partial_3 }{ \wedge^3 \cg(1) } \right)
    = [7] + [61] + [52] + [43] + [41] + [32] + [3] + [21]. 
\end{equation}
All the components above are contained in \(\wedge^3 \cg(1)\). 
Therefore the space 
\begin{equation*}
    H_3 (\cgplus)_3 = \Kernel \Big( \partial_3 |_{\wedge^3 \cg(1)} \Big)
\end{equation*}
consists of the remaining components. 
In other words, the following holds. 
\begin{corollary}
    \(H_3 (\cgplus)_3 = 
    [711] + [63] + [531] + [333] + [52] + [421] + [322]
    + [41] + 2 [311] + 2 [3] \) if \(g \geq 4\). 
\end{corollary}

\par \ 
\par \textsc{Graduate School of Mathematical Sciences, the University of Tokyo, 3-8-1 Komaba, Meguro-ku, Tokyo, 153-8914, Japan}
\par \textit{E-mail address}: \texttt{harako@ms.u-tokyo.ac.jp}
\end{document}